\documentclass[11pt]{amsart}

\usepackage{amsfonts, amstext, amsmath, amsthm, amscd, amssymb}
\usepackage{epsfig, graphics, psfrag, overpic, color, colonequals}

\let\oldmarginpar\marginpar
\renewcommand\marginpar[1]{\oldmarginpar[\raggedleft\footnotesize #1]%
{\raggedright\footnotesize #1}}

\renewcommand{\setminus}{{\smallsetminus}}

\newcommand{\ZZ}{{\mathbb{Z}}}

\newcommand{\abs}[1]{{\left\vert #1 \right\vert}}

\def\ba{\begin{array}}
\def\ea{\end{array}}
\def\bp{\begin{pmatrix}}
\def\ep{\end{pmatrix}}
\theoremstyle{plain}
\newtheorem{theorem}{Theorem}[section]
\newtheorem{corollary}[theorem]{Corollary}
\newtheorem{lemma}[theorem]{Lemma}

\newtheorem{prop}[theorem]{Proposition}

\newtheorem*{namedtheorem}{\theoremname}
\newcommand{\theoremname}{testing}
\newenvironment{named}[1]{\renewcommand{\theoremname}{#1}\begin{namedtheorem}}{\end{namedtheorem}}

\theoremstyle{definition}
\newtheorem{define}[theorem]{Definition}
\newtheorem{remark}[theorem]{Remark}

\newtheorem{question}[theorem]{Question}

\begin{document}
\title{Cosmetic crossings and Seifert matrices}
\author[C. Balm]{Cheryl Balm}
\author[S. Friedl]{Stefan Friedl}
\author[E. Kalfagianni]{Efstratia Kalfagianni}
\author[M. Powell]{Mark Powell}
\address[]{Department of Mathematics, Michigan State University,  E. Lansing 48824}
\email[]{balmcher@math.msu.edu}
\address[]{Mathematisches Institut\\ Universit\"at zu K\"oln\\   Germany}
\email[]{sfriedl@gmail.com}

\address[]{Department of Mathematics, Michigan State University, E. Lansing 48824}
\email[]{kalfagia@math.msu.edu}

\address[]{Department of Mathematics, Indiana University, Bloomington, IN 47405}
\email[]{macp@indiana.edu}

\thanks{ C. B. was supported in part by NSF grant DMS--0805942  and by
NSF/RTG grants DMS-0353717 and  DMS-0739208.}
\thanks{E. K. was supported in part by NSF grants DMS--0805942 and DMS-1105843.}

\begin{abstract} We study cosmetic crossings in knots of genus one and obtain
obstructions to such crossings in terms of knot invariants determined by Seifert matrices.
In particular,  we prove
 that for genus one knots the Alexander polynomial and the homology of the double cover branching over the knot
provide obstructions to cosmetic crossings. As an application we prove the nugatory crossing
conjecture for twisted Whitehead doubles of non-cable knots. We also verify
the conjecture
for several families of pretzel knots and all genus one knots with up to 12 crossings.




\end{abstract}

\maketitle



\section{Introduction}\label{sec:intro}
A fundamental open question in knot theory is the question of when a crossing
change
on an oriented knot changes the isotopy class of the knot. A crossing disc for an oriented knot $K\subset S^3$
is an embedded disc $D\subset S^3$
such
that $K$ intersects ${\rm int}(D)$ twice with
zero algebraic intersection number. A crossing change on $K$ can be achieved
by twisting $D$ or equivalently by performing appropriate Dehn surgery of $S^3$
along the crossing circle $\partial D$.
The crossing is called nugatory if and only if
$\partial D$ bounds an embedded disc in the complement of $K$. A non-nugatory crossing on a knot $K$ is called
cosmetic if the oriented knot $K'$ obtained
from $K$ by changing the crossing  is isotopic to $K$.
Clearly, changing a nugatory crossing does not
change the isotopy class of a knot. The nugatory crossing conjecture (Problem 1.58  of  Kirby's list \cite{Kirbylist}) asserts that the converse is true: if a crossing change on a knot $K$ yields a knot
isotopic to $K$ then the crossing is
nugatory. In other words, there are not any knots in $S^3$ that admit cosmetic crossings.

In the case that $K$ is the trivial knot an affirmative answer follows from a result
of Gabai \cite{gabai} and work of Scharlemann and Thompson \cite{st}. The conjecture is also known  to hold for  2-bridge knots by work of Torisu \cite{torisu},
and for fibered knots by work  of Kalfagianni \cite{kalfagianni}.
For knots of braid index three a weaker form of the conjecture, requiring  that the crossing change  happens on a
closed 3-braid diagram, is discussed by Wiley in \cite{3braids}.

In this paper we study cosmetic crossings on genus one knots and we
show  that  the Alexander polynomial and the homology of the double cover branching over the knot
provide obstructions to cosmetic crossings.

\begin{theorem} \label{general} Given  an oriented genus one knot $K$  let  $\Delta_K(t)$ denote the Alexander polynomial of $K$
and let $Y_K$  denote the double
cover of $S^3$ branching over $K$.
Suppose that  $K$ admits a cosmetic crossing.  Then
\begin{enumerate}
\item $K$ is algebraically slice. In particular,
 $\Delta_K(t) \doteq f(t) f(t^{-1})$,
where $f(t)\in \ZZ[t]$ is a linear polynomial.

\item The homology group $H_1(Y_K):=H_1(Y_K, {\ZZ})$ is  a finite cyclic group.
\end{enumerate}
\end{theorem}
For knots that admit unique (up to isotopy) minimal genus Seifert surfaces we have the following stronger result.
\begin{named}{Theorem \ref{Thm:unique_minimal_genus}}
Let $K$ be an oriented genus one knot with a unique minimal genus Seifert surface, which admits a cosmetic crossing.  Then $\Delta_K(t) \doteq 1$.
\end{named}

Given a knot $K$ let $D_{+}(K, n)$ denote the $n$-twisted, positive-clasped Whitehead double  of $K$
 and let $D_{-}(K, n)$ denote the $n$-twisted, negative-clasped Whitehead double  of $K$.
 Theorems \ref{general} and \ref{Thm:unique_minimal_genus} can be used to prove the nugatory crossing conjecture for several classes of Whitehead doubles.
 For example, Theorem \ref{Thm:unique_minimal_genus}, combined with results of Lyon and Whitten \cite{Lyons, whitten},
 gives  the following. (See Section \ref{sec:examples} for more results in this vein.)
 \begin{corollary} If $K$ is a non-cable knot, then, for every $n\neq 0$,  $D_{\pm }(K, n)$ admits no cosmetic crossings.
 \end{corollary}

Combining Theorem \ref{general}
with a result of Trotter \cite{trotter}, we prove the nugatory crossing conjecture for all the genus one knots with
up to twelve crossings (Theorem \ref{12crossings}) and for several families of pretzel knots (Corollary \ref{pretzels}).
\smallskip

The paper is organized as follows:
In Section \ref{sec:mimimum} we use a result of Gabai \cite{gabai} to prove that a cosmetic crossing change on a knot $K$ can be realized by twisting along an
 essential arc on a minimal genus Seifert surface of $K$ (Proposition \ref{prop:minimum}). For genus one knots such an arc will be non-separating on the surface.
 In subsequent sections this will be our
 starting point for establishing connections between cosmetic crossings and knot invariants determined by Seifert matrices.

 Sections \ref{sec:aslice} and \ref{sec:doublecover} are devoted to the proof of Theorem \ref{general}.
 The proof of this theorem shows that the $S$-equivalence class of the Seifert matrix for a genus one knot provides
 more refined obstructions to cosmetic crossings:
 \begin{corollary}\label{sequivalent}
Let $K$ be a genus one knot. If $K$ admits a cosmetic crossing, then $K$ has a Seifert matrix $V$ of the form
$\begin{pmatrix}a & b \\ b+1 & 0\end{pmatrix}$
which is $S$--equivalent to
$\begin{pmatrix}a +\epsilon & b \\ b+1 & 0\end{pmatrix}$
for some $\epsilon\in \{-1,1\}$.
\end{corollary}

In Section \ref {Section:S_equivalence} we study the question of whether
the  $S$--equivalence class of the Seifert matrix of a genus one knot
 contains enough information to resolve the nugatory crossing conjecture (Question \ref{qu:sequivalent}). Using Corollary \ref{sequivalent} we prove Theorem \ref{Thm:unique_minimal_genus} which implies
 the nugatory crossing conjecture for  genus one knots with non--trivial Alexander polynomial with a \emph{unique} minimal genus Seifert surface. We also construct examples showing that  Corollary \ref{sequivalent} is not enough to prove the nugatory crossing conjecture
for \emph{all}
genus one knots with non--trivial Alexander polynomial (Proposition \ref{prop:sequivalent}).

   In Sections \ref{sec:lowcrossings} and \ref{sec:examples} we provide examples of knots for which Theorems \ref{general} and
\ref{Thm:unique_minimal_genus} settle the nugatory crossing question. In Section  \ref{sec:lowcrossings}
we combine Theorem \ref{general} and Corollary \ref{sequivalent} with a result of Trotter to
settle the conjecture for all the 23 genus one knots with up to 12 crossings.
The examples we discuss in Section \ref{sec:examples}
 are twisted Whitehead doubles and pretzel knots.
\vskip 0.04in

 Throughout the paper we will discuss oriented knots in an oriented $S^3$ and we work in the smooth category.

\subsection*{Acknowledgement} CB and EK thank Matt Hedden and Matt Rathbun for helpful discussions.
Part of this work completed while MP was a visitor at WWU M\"{u}nster, which he thanks for its hospitality.
\smallskip

\section{Crossing changes and arcs on surfaces} \label{sec:mimimum}
 In this section we use a result of Gabai \cite{gabai} to prove that a cosmetic crossing change on a knot $K$ can be realized by twisting along an
essential arc on a minimal genus Seifert surface of $K$ (Proposition \ref{prop:minimum}). For genus one knots such an arc will be non-separating on the surface.
in the next sections this will be our
 starting point for establishing connections between cosmetic crossings and knot invariants determined by Seifert matrices.
 \vskip 0.1in

Let $K$ be an oriented  knot in $S^3$ and $C$ be a crossing
of sign $\epsilon$, where $\epsilon=1$ or $-1$ according to whether $C$ is a
positive or negative crossing (see Figure 1).
A \emph{crossing disc} of $K$ corresponding to $C$
is an embedded disc $D\subset S^3$
such
that $K$ intersects ${\rm int}(D)$ twice, once for each branch of $C$, with
zero algebraic intersection number. The boundary $L = \partial D$ is called a
\emph{crossing circle}.
Performing $({\textstyle {{-\epsilon}}})$-surgery on $L$,
changes $K$ to another knot $K^{'}\subset S^3$ that
is obtained from $K$
by changing the crossing $C$.

\begin{define} \label{nugat} A crossing  supported on a crossing circle $L$
of an oriented knot $K$ is called  \emph{nugatory} if
$L = \partial D$ also bounds an embedded disc in the complement of $K$. This disc and $D$ form an
embedded
2-sphere that decomposes $K$ into a connected sum
where some of the summands may be trivial.
A non-nugatory crossing on a knot $K$ is called
\emph{cosmetic} if the oriented knot $K'$ obtained
from $K$ by changing $C$ is isotopic to $K$; that is, there exists an orientation-preserving diffeomorphism
$f: S^3\longrightarrow S^3$ with $f(K)=K'$.
\qed\end{define}
\vskip .1in

\begin{figure}
\includegraphics[width=1.3in]{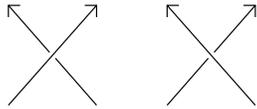}
\caption{Left: a positive crossing. Right: a negative crossing.}\label{cr_signs}
\end{figure}

For a link $J$ in $S^3$ we will use $\eta(J)$ to denote a regular neighborhood of $J$ in $S^3$
and we will use $M_J \colonequals \overline {S^3\setminus \eta(J)}$ to denote the closure of the complement of $\eta(J)$ in $S^3$.

\begin{lemma} \label{lem:irreducible} Let $K$ be an oriented knot and $L$ a crossing circle supporting a crossing $C$ of $K$.
Suppose that $M_{K \cup L}$ is reducible. Then $C$ is nugatory.

\end{lemma}
\begin{proof} An essential 2-sphere in $M_{K \cup L}$ must separate
$\eta(K)$ and $\eta(L)$. Thus in $S^3$,  $L$ lies in a 3-ball disjoint from $K$.
Since $L$ is unknotted, it bounds a disc in the complement of $K$.
 \end{proof}

Let $K$ be an oriented knot and $L = \partial D$ a crossing circle supporting a  crossing $C$.
Let $K'$ denote the knot obtained from $K$ by changing $C$.
Since the linking number of $L$ and $K$ is zero, $K$ bounds a Seifert surface in the complement of $L$. Let
$S$ be a Seifert surface that is of minimal genus among all such Seifert surfaces in the complement of $L$. Since $S$ is incompressible, after an isotopy we can arrange so that the closed components of  $S\cap D$ are homotopically essential in $D\setminus K$.
But then each such component is parallel to $\partial D$ on $D$  and by further modification we can arrange so that
$S\cap D$ is a single arc $\alpha$  that is properly embedded on $S$ as illustrated in Figure \ref{alpha}.
The surface
$S$
gives rise
to Seifert surfaces $S$ and $S'$ of $K$ and $K'$, respectively.
\begin{figure}
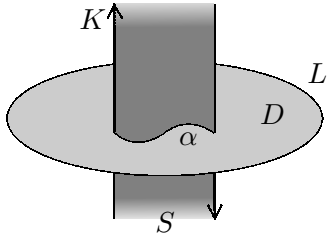
\caption{The crossing arc $\alpha = S \cap D$.}\label{alpha}
\end{figure}

\begin{prop} \label{prop:minimum}
Suppose that $K$ is isotopic to $K'$. Then $S$ and $S'$ are Seifert surfaces of minimal genus for $K$ and $K'$, respectively.
\end{prop}
\begin{proof} If the crossing is nugatory then $L$ bounds a disc in the complement of $S$ and the conclusion is clear.
Suppose the crossing is cosmetic; by Lemma \ref{lem:irreducible},
$M_{K \cup L}$ is irreducible.
We can consider the surface $S$ properly embedded in $M_{K \cup L}$
so that it is disjoint from
$\partial \eta(L) \subset \partial M$.
The assumptions
on irreducibility of $M_{K \cup L}$ and
on the genus of $S$ imply that the foliation machinery of Gabai \cite{gabai} applies. In particular,
$S$ is taut in the Thurston norm of $M_{K\cup L}$. The manifolds
$M_K$ and $M_{K'}$ are obtained by Dehn filling of $M_{K \cup L}$ along $\partial \eta(L)$.
By \cite[Corollary 2.4] {gabai}, $S$ can fail to remain taut in the Thurston norm
(i.e. genus minimizing)
in at most one of $M_K$ and $M_{K'}$.   Since we have assumed that $C$ is a cosmetic crossing,  $M_K$ and $M_{K'}$
are homeomorphic (by an orientation-preserving homeomorphism). Thus $S$
remains taut in both of $M_K$ and $M_{K'}$. This implies that $S$
and $S'$ are Seifert surfaces of minimal genus for $K$ and $K'$, respectively.
\end{proof}

By Proposition \ref{prop:minimum}, a crossing change of a knot $K$ that produces an isotopic knot
corresponds to a properly embedded arc $\alpha$ on a minimal genus Seifert surface $S$ of $K$.
We observe the following.
\begin{lemma} \label{lem:essential} If $\alpha$ is inessential on $S$, then the crossing is nugatory.
\end{lemma}
\begin{proof} Recall that $\alpha$ is the intersection of a crossing disc $D$ with $S$.
Since $\alpha$ is inessential, it separates $S$ into two pieces, one of which is a disc $E$.
Consider $D$ as properly embedded in a regular neighborhood $\eta(S)$ of the surface $S$.
The boundary of a regular neighborhood of $E$ in $\eta (S)$ is a 2-sphere
that contains the crossing disc $D$. The complement of the interior of $D$ in that 2-sphere gives a disc
bounded by the crossing circle $L=\partial D$ with its interior disjoint from the  knot $K=\partial S$.
\end{proof}

\section{Obstructing cosmetic crossings in genus one knots}  \label{sec:aslice}
A  knot  $K$ is called \emph{algebraically slice}
if it admits a Seifert surface $S$ such that the Seifert form
$\theta: H_1(S)\times H_1(S)\longrightarrow \ZZ$
vanishes on a half-dimensional summand of $H_1(S)$; such a summand is called a \emph{metabolizer} of $H_1(S)$.
If $S$ has genus one, then the existence  of a metabolizer for $H_1(S)$ is equivalent to the existence
of an essential oriented simple closed curve on $S$ that has
zero self-linking number. If $K$ is algebraically slice, then the Alexander polynomial
$\Delta_K(t)$ is of the form $\Delta_K(t) \doteq f(t) f(t^{-1})$,
where $f(t)\in \ZZ[t]$ is a linear polynomial with integer coefficients
and $\doteq$ denotes equality up to multiplication by a unit in the ring of Laurent polynomials ${\ZZ}[t, t^{-1}]$.
For more details on these and other classical knot theory concepts we will use in this and the next section, the reader is referred to
\cite{burde-zieschang:knots} or  \cite{lickorish:book}.

\begin{theorem}\label{aslice} Let $K$ be an oriented genus one knot.
If $K$ admits a cosmetic crossing, then it is algebraically slice. In particular,
there is a   linear polynomial $f(t)\in \ZZ[t]$ such that
 $\Delta_K(t)\doteq f(t) f(t^{-1})$.

\end{theorem}
\begin{proof} Let $K'$ be a knot that is obtained from $K$ by a cosmetic crossing change $C$.
By Proposition \ref{prop:minimum}, there is a genus one Seifert surface $S$
such that a crossing disc supporting $C$ intersects $S$ in a properly embedded arc $\alpha \subset S$.
Let $S'$ denote the result of $S$ after the crossing change.
 \begin{figure}
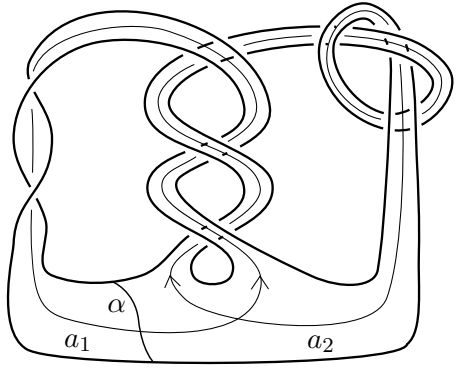
\caption{A genus one surface $S$ with generators $a_1$ and $a_2$ of $H_1 (S)$ and a non-separating arc $\alpha$.}\label{fig1}
\end{figure}
Since $C$ is a cosmetic crossing, by Lemma \ref{lem:essential},  $\alpha$ is essential.  Further, since the genus of $S$ is one, $\alpha$ is non-separating.
We can find a simple closed curve $a_1$ on $S$ that intersects $\alpha$ exactly once. Let $a_2$ be another simple closed curve so that
$a_1$ and $a_2$ intersect exactly once and the homology classes of $a_1$ and $a_2$ form a  symplectic basis for
 $H_1 (S) \cong {\ZZ} \oplus {\ZZ}$.  Note that $\{ a_1, a_2 \}$ form a corresponding basis of $H_1(S')$. See Figure \ref{fig1}.

 The Seifert matrices of $S$  and $S'$ with respect to these bases are
 $$V = \begin{pmatrix}a & b \\ c & d\end{pmatrix} \ \  \ \rm{and}  \ \ \   V'= \begin{pmatrix}a-\epsilon & b \\ c & d\end{pmatrix}$$
respectively, where $a,b,c, d\in {\ZZ}$ and $\epsilon=1$ or $-1$ according to whether $C$ is a positive or a negative crossing.
 The Alexander polynomials of $K$,  $K'$ are given by
 $$\Delta_K(t)\doteq \det( V-tV^T)=ad(1-t)^2-(b-ct)(c-tb),$$
 $$\Delta_{K'}(t)\doteq(a-\epsilon)d(1-t)^2-(b-ct)(c-tb).$$
Since $K$ and $K'$ are isotopic we must have $\Delta_K(t)\doteq\Delta_{K'}(t)$ which easily implies that $d={\rm lk}(a_2, \ a_2)=0$.
Hence $K$ is algebraically slice and
 $$\Delta_K(t)\doteq(b-ct)(c-tb)=(-t)(b-ct)(b-ct^{-1})\doteq (b-ct)(b-ct^{-1}).$$
 Setting $f(t)=b-ct$ we obtain $\Delta_K(t)\doteq f(t) f(t^{-1})$ as desired.  Note that since $\abs{b-c}$ is the intersection number between $a_1$ and $a_2$, by suitable orientation choices, we may assume that $c=b+1$.
\end{proof}

Recall that the determinant of a knot $K$ is defined by $\det(K)=\abs{\Delta_K(-1)}$.
As a corollary of Theorem \ref{aslice} we have the following.

\begin{corollary}\label{determinant} Let $K$ be a genus one knot. If $\det(K)$ is not a perfect square then
$K$ admits no cosmetic crossings.
\end{corollary}
\begin{proof}  Suppose that $K$ admits a cosmetic crossing. By Theorem \ref{aslice}
$\Delta_K(t)\doteq f(t) f(t^{-1})$,
where $f(t)\in \ZZ[t]$ is a linear polynomial. Thus, if $K$ admits cosmetic crossings we have
 $\det(K)=\abs{\Delta_K(-1)}=[f(-1)]^{2}$.
\end{proof}

\section{Further obstructions: homology of double covers} \label{sec:doublecover}
In this section we derive further obstructions to cosmetic crossings in terms of the homology of the double branched cover
of the knot.
More specifically, we will prove the following.

\begin{theorem}\label{doublecover} Let $K$ be an oriented genus one knot and let $Y_K$  denote the double
cover of $S^3$ branching over $K$.
If $K$ admits a cosmetic crossing, then the homology group $H_1(Y_K)$ is  a finite cyclic group.
\end{theorem}

To prove Theorem \ref{doublecover} we need the following elementary lemma.
(Here, given $m\in {\ZZ}$ we denote by  ${\ZZ}_m={\ZZ}/{m{\ZZ}}$  the cyclic abelian group of order $\abs{m}$.)

\begin{lemma} \label{abelian} If $H$ denotes the abelian group given by the presentation
  \begin{equation*}
H\cong \left<
\begin{array}{l|l}
c_1, c_2 & 2x c_1+ (2y+1)c_2= 0\\
& (2y+1)c_1= 0
\end{array}
\right> ,
\end{equation*}
then we have
\begin{enumerate}
\item $H\cong 0$, if $y=0$ or $y=-1$.
\item  $H \cong \mathbb{Z}_{d} \oplus \mathbb{Z}_{\frac{{(2y+1)}^2}{d}}$, if   $y \neq 0,\ -1$ and {\rm gcd}$(2x,\ 2y+1) = d$ where $1 \leq  d \leq 2y+1$.
\end{enumerate}
  \end{lemma}
  \begin{proof}  If $y=0$ or $y=-1$, clearly we have $H\cong 0$.
  Suppose now that $y \neq 0,\ -1$ and {\rm gcd}$(2x,\ 2y+1) = d$ where $1 \leq  d \leq 2y+1$.
  Then there are integers $A$ and $B$ such that $2x = dA,\ 2y+1 = dB$, and $\textrm{gcd}(A,B) = 1$.  Let $\alpha$ and $\beta$ be such that $\alpha A + \beta B = 1$.  Since $ \begin{pmatrix}dA & dB \\ dB & 0\end{pmatrix}$ is a presentation matrix of $H$
  and $\begin{pmatrix}\alpha & \beta \\ -B & A\end{pmatrix}$ is invertible over $\mathbb{Z}$, we get that
$\begin{pmatrix}\alpha & \beta \\ -B & A\end{pmatrix} \begin{pmatrix}dA & dB \\ dB & 0\end{pmatrix} = \begin{pmatrix}d & d \alpha B \\ 0 & -dB^2\end{pmatrix}$ is also a presentation matrix for $H$.  So

\begin{equation*}
H\cong
\left<
\begin{array}{l|l}
c_1,c_2 & dc_1 + d \alpha B c_2 = 0\\
& dB^2 c_2 = 0
\end{array}
\right>
\end{equation*}
Now letting $c_3 = c_1+ \alpha Bc_2$, we have

\begin{equation*}
H \cong
\left<
\begin{array}{l|l}
c_2, c_3 & dc_3 = 0\\
& dB^2 c_2= 0
\end{array}
\right>
\end{equation*}

Hence $H \cong \mathbb{Z}_{d} \oplus \mathbb{Z}_{dB^2}=\mathbb{Z}_{d} \oplus \mathbb{Z}_{\frac{{(2y+1)}^2}{d}}$.
  \end{proof}
\vskip 0,05in

\begin{proof} [Proof of Theorem \ref{doublecover}]
  Suppose that a genus one knot $K$ admits a cosmetic crossing yielding an isotopic knot  $K'$.
The proof of Theorem \ref{aslice} shows that $K$ and $K'$ admit Seifert matrices
of the form

\begin{equation}\label{Eqn:matrix1}V= \left(\ba{cc} a & b \\ b+1 & 0 \ea\right)\mbox{and} \  V'=\left(\ba{cc} a+\epsilon & b \\ b+1 & 0 \ea\right)\end{equation}
respectively, where $a,b \in {\ZZ}$ and $\epsilon=1$ or $-1$ according to whether $C$ is a negative or a positive crossing.
 In particular we have
 \begin{equation}\label{Eqn:matrix2}\Delta_K(t)\doteq \Delta_{K'}(t)\doteq b(b+1)(t^2+1)-(b^2+(b+1)^2).\end{equation}
 Presentation matrices for $H_1(Y_K)$ and $H_1(Y_{K'})$ are given by
 \begin{equation}\label{Eqn:matrix3}V+V^T = \begin{pmatrix}2a & 2b+1\\ 2b+1 & 0\end{pmatrix} \mbox{and} \ V'+  (V')^{T}= \begin{pmatrix} 2a+2\epsilon &2 b+1 \\ 2b+1 & 0\end{pmatrix},\end{equation}
 respectively.
It follows that Lemma \ref{abelian} applies to both $H_1(Y_K)$ and $H_1(Y_{K'})$. By that lemma,
$H_1(Y_K)$ is either cyclic or $H_1(Y_K)\cong\mathbb{Z}_{d} \oplus \mathbb{Z}_{\frac{{(2b+1)}^2}{d}}$,
with   $b \neq 0,\ -1$ and $ {\rm gcd}(2a,\ 2b+1) = d$ where $1 < d \leq 2b+1$.
Similarly, $H_1(Y_{K'})$ is either cyclic or $H_1(Y'_K)\cong\mathbb{Z}_{d'} \oplus \mathbb{Z}_{\frac{{(2b+1)}^2}{d'}}$,
with $ {\rm gcd}(2a+2\epsilon,\ 2b+1) = d'$ where $1 < d' \leq 2b+1$.
Since $K$ and $K'$ are isotopic, we have $H_1(Y_K)\cong H_1(Y_{K'})$. One can easily verify this
can only happen in the case that $\textrm{gcd}(2a,\ 2b+1) = \textrm{gcd}(2a+2\epsilon,\ 2b+1) = 1$ in which case $H_1 (Y_K)$ is cyclic.
 \end{proof}

 It is known that for an algebraically slice knot of genus one every minimal genus surface $S$ contains a metabolizer
 (compare \cite[Theorem 4.2]{livingston}).
 After completing the metabolizer  to a basis of $H_1(S)$ we have a Seifert matrix $V$ as in (\ref{Eqn:matrix1}) above.

 \begin{corollary} \label{matrix} Let $K$ be an oriented, algebraically slice knot of genus one.
 Suppose that a genus one Seifert surface of $K$ contains a metabolizer leading to a Seifert matrix $V$ as in (\ref{Eqn:matrix1})
 so that  $b\neq 0,-1$ and $\emph{gcd}(2a,\ 2b+1) \neq 1$.
 Then $K$ cannot admit a cosmetic crossing.
 \end{corollary}

\begin{proof} Let $d=\textrm{gcd}(2a,\ 2b+1) $. As in the proof of Theorem \ref{doublecover},
 we use Lemma \ref{abelian}  to conclude that  $H_1(Y_K)\cong\mathbb{Z}_{d} \oplus \mathbb{Z}_{\frac{{(2b+1)}^2}{d}}$
 and hence is non-cyclic unless $d=1$. Now the conclusion follows by Theorem \ref{doublecover}.
  \end{proof}

 Theorems \ref{aslice} and \ref{doublecover}  immediately  yield Theorem
 \ref{general} stated in the introduction.

 \section{$S$--equivalence of Seifert matrices}\label{Section:S_equivalence}
 We begin by recalling the notion of $S$-equivalence.

\begin{define} \label{defi:sequivalent} We say that an integral square matrix $V$ is a \emph{Seifert matrix} if $\det(V-V^T)=1$.
We say that two Seifert matrices are \emph{$S$--equivalent} if they are related by a finite sequence of the following moves or their inverses:
\begin{enumerate}
\item replacing $V$ by $PVP^T$, where $P$ is an integral unimodular matrix,
\item column expansion, where we replace  an $n \times n$ Seifert matrix $V$ with an $(n+2) \times (n+2)$ matrix of the form:
\[\left(\begin{array}{ccc|cc}
    & & &  0 & 0\\
    & V & & \vdots & \vdots\\
    & & & 0 & 0\\ \hline
    u_1 & \cdots & u_n & 0 & 0\\
    0 & \cdots & 0 & 1 & 0
  \end{array}\right),\]
where $u_1,\dots,u_n \in \ZZ$,
\item a row expansion, which is defined analogously to the column expansion, with the r\^{o}les of rows and columns reversed.
\end{enumerate}
\qed\end{define}

Note that $W$ is a row expansion of $V$ if and only if $W^T$ is a column expansion of $V^T$.
In the following, given two Seifert matrices $V$ and $W$ we write $V\sim W$ if they are $S$--equivalent, and we write $V \approx W$ if they are congruent.

The proof of Theorem \ref{aslice} immediately gives Corollary  \ref{sequivalent} stated in the introduction. This in turn leads to the following question.

\begin{question}\label{qu:sequivalent}
Let $a, b$ and $d$ be integers and $\epsilon\in \{-1,1\}$. Are the matrices
\[ \bp a&b \\ b+1&d\ep\mbox{ and }
  \bp a+\epsilon&b \\ b+1&d\ep \]
$S$--equivalent?
\end{question}

Now we focus on Question \ref{qu:sequivalent}. A first trivial observation is that if $d=0$ and $b=0$, then the two given matrices are congruent and, in particular, $S$--equivalent.
We therefore restrict ourselves to matrices with non--zero determinant, or equivalently, to knots of genus one such that the Alexander polynomial $\Delta_K(t)=\det(V-tV^T)$ is non--trivial.

\subsection{Knots with a unique minimal genus Seifert surface} In this subsection we prove an auxiliary algebraic result about congruences  of
Seifert matrices.
As a first application of it we prove the nugatory crossing conjecture for  genus one knots with non--trivial Alexander polynomial and with a minimal genus Seifert surface which, up to isotopy, is \emph{unique}.
\begin{prop}\label{prop:general_congruences}
Suppose that the matrices
 \[ \bp a&b \\ b+1&0\ep\mbox{ and }
  \bp c&b \\ b+1&0\ep, \]
where $a,b,c \in \ZZ$, are congruent over $\ZZ$.  Then there is an integer $n$ such that $a + n(2b+1) = c$.
\end{prop}

Before we prove Proposition \ref{prop:general_congruences} we explain one of  its consequences.
If $K$ is a knot with, up to isotopy, a unique minimal genus Seifert surface, then the Seifert matrix corresponding to that surface only depends on the choice of basis for the first homology.  Put differently, the integral congruence class of the Seifert matrix corresponding to the unique minimal genus Seifert surface is an invariant of the knot $K$.  Assuming Proposition \ref{prop:general_congruences}, we have the following theorem.

\begin{theorem}\label{Thm:unique_minimal_genus}
Let $K$ be an oriented genus one knot with a unique minimal genus Seifert surface, which admits a cosmetic crossing.  Then $\Delta_K(t) \doteq 1$.
\end{theorem}

\begin{proof}
Let $K$ be a genus one knot with a unique minimal genus Seifert surface, which admits a cosmetic crossing.
It follows from Corollary  \ref{sequivalent} and from the discussion preceding the statement of this theorem that
$K$
 admits a Seifert matrix
 $\begin{pmatrix}a & b \\ b+1 & 0\end{pmatrix}$
which is $S$--equivalent to
$\begin{pmatrix}a +\epsilon & b \\ b+1 & 0\end{pmatrix}$
for some $\epsilon\in \{-1,1\}$.
 For $b\neq 0$, Proposition \ref{prop:general_congruences} precludes such congruences from being possible.  If $b=0$, then the Alexander polynomial is $1$.
\end{proof}

We now proceed with the proof of  Proposition \ref{prop:general_congruences}.
\begin{proof}[Proof of  Proposition \ref{prop:general_congruences}]
To begin, we suppose that an integral unimodular congruence exists as hypothesized. That is, suppose that there exist integers $x,y,z,t$ such that
\[\left(\ba{cc} x & y \\ z & t \ea \right)\left(\ba{cc} a & b \\ b+1 & 0 \ea \right)\left(\ba{cc} x & z \\ y & t \ea \right) = \left(\ba{cc} c & b \\ b+1 & 0 \ea \right).\]
The left hand side multiplies out to give
\begin{equation}\label{Eqn:matrix4} \left(\ba{cc} x^2a+xy(2b+1) & xza+yz(b+1)+xtb \\ xza + xt(b+1) +zyb  & z^2a+zt(2b+1) \ea\right).\end{equation}
Solving the bottom right entry equal to zero implies that either $z=0$, or (for $a \neq 0$) $z = -t(2b+1)/a$.  We required that $z$ was an integer, so it must also be the case that $t$ is such that $a$ divides $t(2b+1)$, but we shall not need this.  In the case that $a = 0$ and $z \neq 0$, we then have $t=0$.

First, if $z=0$, then (\ref{Eqn:matrix4}) becomes
\[\left(\ba{cc}x^2a +(2b+1)xy & xtb \\ xt(b+1) & 0 \ea\right).\]
We require that $x=t=1$ or $x=t=-1$ for the top right and bottom left entries to be correct.  But then setting $n=xy$ proves the proposition in this case.

Next, suppose $z\ne 0$ and $a=0$.  Then $t=0$, and (\ref{Eqn:matrix4}) becomes
\[\left(\ba{cc} (2b+1)xy & yz(b+1) \\ zyb  & 0 \ea\right).\]
The equations $zyb = b+1$ and $zy(b+1) = b$ imply that $b^2 = (b+1)^2$, which has no integral solutions.

Now in the general case, i.e. $z\ne 0$ and $a\ne 0$, we substitute $z = -t(2b+1)/a$ into (\ref{Eqn:matrix4}), to yield
\[\left(\ba{cc} xk & -t(b+1)k/a \\ -tbk/a & 0 \ea\right),\]
where $k := ax+y(2b+1)$.  Setting this equal to $$\left(\ba{cc} c & b \\ b+1 & 0 \ea \right),$$ the equations
\[-t(b+1)k/a = b\]
and
\[-tbk/a = b+1\]
imply again that $(b+1)^2 = b^2$.  Since this does not have integral solutions, we also rule out this case.  The only congruences possible are therefore those claimed, which occur when $z=0$ and $x=t= \pm 1$.  This completes the proof of Proposition \ref{prop:general_congruences}.
\end{proof}

\subsection{Other algebraically slice genus one knots}
In this subsection we will show that, in general, the answer to Question \ref{qu:sequivalent} can be affirmative, even for matrices with non--zero determinant. This implies that the  $S$--equivalence class of the Seifert matrix of a genus one knot with non--trivial Alexander polynomial does not in general contain enough information to resolve the nugatory crossing conjecture.
In fact we will prove the following proposition:

\begin{prop}\label{prop:sequivalent}
For any $b>4$ such that $b \equiv 0 \text{ or } 2\mod3$, there exists an $a\in \ZZ$ such that
$$V = \begin{pmatrix}a & b \\ b+1 & 0\end{pmatrix} \ \  \ \rm{and}  \ \ \   V'= \begin{pmatrix}a+1& b \\ b+1 & 0\end{pmatrix}$$are $S$--equivalent.
\end{prop}

 Since any Seifert matrix $V$ can be realized as the Seifert matrix of a knot it follows  that the $S$--equivalence class of Seifert matrices cannot resolve the nugatory crossing conjecture for genus one knots with non--trivial Alexander polynomial.

We will need the following elementary lemma to prove Proposition \ref{prop:sequivalent}.

\begin{lemma}
Let $a,b,k\in \ZZ$, then the matrices
\[  \bp a&b \\ b+1&0\ep, \bp a+k(2b+1)&b \\ b+1&0\ep  \mbox{ and }  \bp ab^2&b \\ b+1&0\ep \]
are $S$--equivalent.
\end{lemma}

\begin{proof}
It is obvious that the first two matrices are congruent. It remains to show that the first and the third matrix are $S$--equivalent.
This follows immediately from the following sequence of $S$--equivalences:
\[ \ba{rll} & \bp a&b \\ b+1&0\ep \Rightarrow \bp a&b& 1&0 \\ b+1&0&0&0 \\ 0&0&0&1 \\ 0&0&0&0 \ep
\Rightarrow \bp a&0&1&0 \\ b+1&0&0&-b  \\ 0&0&0&1 \\ 0&0&0&0 \ep
\\ & \\
\Rightarrow&
\bp a&0&1 &0 \\ 0&0&0&0 \\ 0&1&0&0 \\ b+1&-b&0&0 \ep
\Rightarrow \bp a&0&1&0 \\ 0&0&0&0 \\ 1&1&0&0 \\ 1&-b&0&0 \ep
\Rightarrow \bp a&ab&1&0 \\ ab&ab^2&b&0 \\ 0&1+b&0&0 \\ 1&0&0&0 \ep
\\ & \\
\Rightarrow& \bp 0&1+b&0&0 \\ b&ab^2&ba&0 \\ 1&ab&a&0 \\ 0&0&1&0\ep
\Rightarrow \bp 0&1+b \\ b&ab^2 \ep \Rightarrow \bp ab^2 & b\\ b+1&0\ep.\ea \]
\end{proof}

Using this lemma we can now prove the proposition:

\begin{proof}[Proof of Proposition \ref{prop:sequivalent}]
Let $b>4$ and such that $b \equiv 0 \text{ or } 2\mod3$. It is then straight--forward to see $1+b$ is coprime to $2(b+1)-1=2b+1$ and that
$b-1$ is coprime to $2(b-1)+3$. In particular $1-b^2=(1-b)(1+b)$ is coprime to $2b+1$.
We can therefore find an $a\in \ZZ$ such that
\[ a(1-b^2)\equiv -1\mod (2b+1),\]
by the Chinese Remainder Theorem.  Put differently, we can find a $k\in \ZZ$ such that
\[ a+1=ab^2+k(2b+1).\]
It follows from the above lemma that
\[   \bp a&b \\ b+1&0\ep\mbox{ and } \bp a+1&b \\ b+1&0\ep \]
are $S$--equivalent.
 \end{proof}

\section{Low crossing knots} \label{sec:lowcrossings}
In this section we combine Theorem \ref{general} and  Corollary \ref{sequivalent}
with the following result  of Trotter \cite{trotter}
to prove the nugatory crossing conjecture for all genus one knots with up to 12 crossings.

\begin{theorem} \cite[Corollary~4.7]{trotter}\label{thm:trotter_congruence}
Let $V$ be a Seifert matrix with $|\det(V)|$ a prime or $1$.  Then any matrix which is $S$-equivalent to $V$ is congruent to $V$ over $\ZZ$.
\end{theorem}

\begin{theorem}\label{12crossings}Let $K$ be a genus one knot that has a diagram with at most 12 crossings.
Then $K$ admits no cosmetic crossings.
\end{theorem}
\begin{proof} Table 1,  obtained from KnotInfo \cite{knotinfo}, gives  the 23 knots of genus one with at most 12 crossings,
with the values of their determinants.
We observe that there  are four knots with square determinant. These are $6_1$,  $9_{46},  10_3$ and $11\textrm{n}_{139}$
which are all known to be algebraically slice.
Thus Corollary  \ref{determinant} excludes
cosmetic crossings for all but these four knots.
Now  $6_1$ and $10_3$ are 2-bridge knots; by \cite{torisu} they do not admit cosmetic crossings.
The knot
$K=9_{46}$  is isotopic to the pretzel knot $P(3,3,-3)$ of Figure \ref{pretzelknots} which
has Seifert matrix $ \begin{pmatrix}3 & 2\\ 1 & 0\end{pmatrix}$ since
 the pretzel knot $P(p,q,r)$ has a Seifert matrix given by $\frac{1}{2}\begin{pmatrix}{p+q} & {q+1} \\ {q-1} & q+r\end{pmatrix}$; see \cite[Example 6.9]{lickorish:book}.  The homology $H_1(Y_K)$ is represented by $ \begin{pmatrix}6 & 3\\ 3 & 0\end{pmatrix}$
(compare Corollary
\ref{pretzels} below). Thus by Lemma \ref{abelian}, $H_1(Y_K)\cong {\ZZ}_3\oplus{\ZZ}_3$, and by Theorem \ref{doublecover},
$K$
 cannot have cosmetic crossings.
 \begin{table}[height=1.00in]
\begin{center}
\begin{tabular}{|c|c|c|c|c|c|}
\hline
$K$& $\det (K)$ & $K$ & $\det (K)$& $K$ & $\det (K)$  \\
\hline
$3_1$ &3& $9_2$ &15& ${11\textrm{a}}_{362}$ & 39 \\
\hline
$4_1$ & 5& $9_5$ & 23&${11\textrm{a}}_{363}$&35 \\
\hline
$5_2$ &7&$9_{35}$ &27&${\bf {11\textrm{\bf n}}_{139}}$ &{\bf 9}\\
\hline
${\bf 6_1}$ &{\bf 9}&${ \bf 9_{46}}$ &{\bf 9}&${11\textrm{n}}_{141}$& 21\\
\hline
$7_2 $& 11&$10_1$&17&${12\textrm{a}}_{803}$ & 21\\
\hline
$7_4$&15&${\bf 10_3}$ & {\bf 25}&${12\textrm{a}}_{1287}$ & 37\\
\hline
$8_1$ &13&${11\textrm{a}}_{247}$ &19&${12\textrm{a}}_{1166}$& 33\\
\hline
$8_3$&17&${11\textrm{a}}_{343}$ & 31&-&- \\
\hline
\end{tabular}
\label{twelve}
\end{center}
\vskip.13in
\caption{Genus one knots with at most 12 crossings.}
\end{table}

The only remaining knot from Table  1 is  the knot $K=11\textrm{n}_{139}$.  This knot is isotopic to the pretzel knot $P(-5,3,-3)$.
There is therefore a genus one surface
 for which
a Seifert matrix is
\[V = \left(\ba{cc} -1 & 2 \\ 1 & 0 \ea \right),\]
again by \cite[Example 6.9]{lickorish:book}.
Using this Seifert matrix we calculate
 $H_1(Y_K)$ $\cong {\ZZ}_9$.
 Thus Theorem \ref{general} does not work for  the knot $11\textrm{n}_{139}$.
Next we turn to Corollary \ref{sequivalent}.
 Since $|\det(V)| = 2$ is prime, by Theorem \ref{thm:trotter_congruence} it suffices to show that $V$ is neither integrally congruent
to
\[\left(\ba{cc} 0 & 2 \\ 1 & 0 \ea \right)\mbox{ nor to } \left(\ba{cc} -2 & 2 \\ 1 & 0 \ea \right).\]
But this follows from Proposition \ref{prop:general_congruences}, with $a=-1$ and $b=1$, noting that two matrices are congruent if and only if their transposes are.
\end{proof}

\begin{remark}
 The method applied for $11\textrm{n}_{139}$ in the proof of
 Theorem \ref{12crossings}
  can also be used to show that the knots $6_1$ and $9_{46}$ do not admit cosmetic crossings.
\end{remark}

 \section{More examples} \label{sec:examples}
 In this section we discuss some families of examples for which  Theorems \ref{general} and \ref{Thm:unique_minimal_genus}
 imply the nugatory crossing conjecture.

\subsection{Twisted Whitehead doubles} Given a knot $K$ let $D_{+}(K, n)$ denote the $n$-twisted
Whitehead double of $K$ with a positive clasp and let  $D_{-}(K, n)$ denote the $n$-twisted
Whitehead double of $K$ with a negative clasp.

\begin{corollary}\label{whitehead}\  (a)\
Given a knot $K$,  the Whitehead double $D_{+}(K, n)$ admits no cosmetic crossing if either $n<0$ or $\abs{n}$ is odd.  Similarly  $D_{-}(K, n)$ admits no cosmetic crossing if either $n>0$ or $\abs{n}$ is odd.
\vskip 0.03in

(b)\  If $K$ is not a cable knot then $D_{\pm}(K, n)$ admits  no cosmetic crossings for every $n\neq 0$.
\end{corollary}

\begin{proof}  (a)\ A Seifert surface of $D_{+}(K, n)$ obtained by plumbing an $n$-twisted annulus
with core $K$ and a Hopf band gives rise to a  Seifert matrix
$V_n = \begin{pmatrix}-1& 0 \\ -1 & n\end{pmatrix}$  \cite[Example 6.8]{lickorish:book}.
Thus the Alexander polynomial
is
of the form
 \begin{equation}\label{Eqn:matrix5}\Delta_K(t)\doteq \Delta_{K'}(t)\doteq -n(t^2+1)+(1+2n)t=:\Delta_n.\end{equation}
Suppose now that
$D_{+}(K, n)$ admits a cosmetic crossing.  Then
$\Delta_n$ should be of the form shown in equation (\ref{Eqn:matrix2}).
Comparing the leading coefficients in the expressions (\ref{Eqn:matrix2}) and (\ref{Eqn:matrix5})
we obtain $\abs{n}=\abs{b(b+1)}$ which implies that $\abs{n}$ should be even.
We have shown that if $\abs{n}$ is odd then $D_{+}(K, n)$ admits no
cosmetic crossing changes.
Suppose now that $n<0$. Since the Seifert matrix $V_n$
depends only on $n$ and not on $K$,  the knot $D_{+}(K, n)$ is $S$-equivalent
to the $n$-twisted, positive-clasped  double of the unknot. This is a positive knot
(all the crossings in the standard diagram of $D_{+}(O, n)$ are positive)
and it has non-zero signature \cite{positive}. Hence $D_{+}(K, n)$
is not algebraically slice and by Theorem \ref{aslice} it cannot admit
cosmetic crossings.

A similar argument holds for $D_{-}(K, n)$.

\vskip 0.05in

(b)\   Suppose that $K$ is not a cable knot. By results of Lyon and Whitten \cite{Lyons, whitten}, for every $n \neq 0$
the Whitehead doubles $D_{\pm}(K, n)$ have unique Seifert surfaces of minimal genus.
By (\ref{Eqn:matrix5}), $\Delta_n\neq 1$, and the conclusion follows by Theorem \ref{Thm:unique_minimal_genus}.
\end{proof}
\begin{figure}
  \includegraphics[width=1.5in]{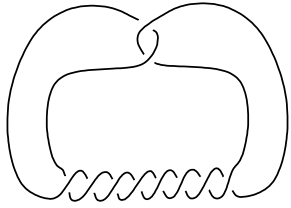}
  \caption{The  $(-4)$-twisted  negative-clasped double of the unknot, $D_{-}(O, -4)$.}
  \label{fig:twistknot}
\end{figure}

\begin{figure}
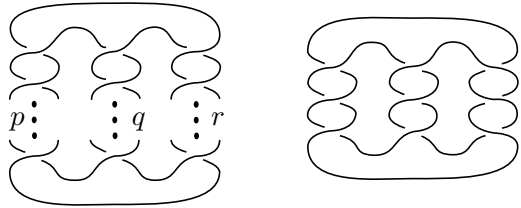
\caption{$P(p,q,r)$ with $p,q$ and $r$ positive and $ P(3,3,-3)$.}
\label{pretzelknots}
\end{figure}

\subsection{Pretzel knots}
Let $K$ be a three string  pretzel knot $P(p,q,r)$ with $p,q$ and $r$  odd (see Figure \ref{pretzelknots}).  The knot  determinant is given
by $\det (K) = |pq+qr+pr|$ and if $K$ is non-trivial then it has genus one.
It is known that  $K$ is algebraically  slice  if and only if $pq+qr+pr=-m^2$, for some odd  $m\in {\ZZ}$ \cite{levine}.

\begin{corollary} \label{pretzels}The knot  $P(p,q,r)$ with $p,q$ and $r$ odd does not admit cosmetic crossings if one of the following is true:
\vskip 0.05in

(a)\  $pq+qr+pr\neq -m^2$, for every odd $m\in {\ZZ}$.
\vskip 0.05in

(b)\ $q+r=0$ and $\textrm{gcd}(p,\  q)\neq 1$.
\vskip 0.05in

(c)\  $p+q=0$ and $\textrm{gcd}(q,\  r)\neq 1$.

\end{corollary}

\begin{proof} In case (a) the result follows from Theorem \ref{aslice} and the discussion above.
For case (b) recall that there is a genus one surface for $P(p,q,r)$ for which
a Seifert matrix is $V_{(p,q,r)} = \frac{1}{2}\begin{pmatrix}{p+q} & {q+1} \\ {q-1} & q+r\end{pmatrix}$ \cite[Example 6.9]{lickorish:book}.
Suppose that   $q+r=0$. If  $\textrm{gcd}(p,\  q)\neq 1$,
then  $\textrm{gcd}(p+q,\  q)\neq 1$ and the conclusion in case (b) follows by
Corollary \ref{matrix}. Case (c) is similar.
\end{proof}

\bibliographystyle{annotate}
\bibliography{biblio}

\end{document}